\documentclass[12pt]{article}

\usepackage{amsmath, amsthm, amssymb, mathrsfs}

\usepackage{fullpage}
\usepackage{bbm}
\usepackage{tikz}
\usepackage{stmaryrd}
\usepackage{mathrsfs}
\usepackage[colorlinks=true,
linkcolor=blue,citecolor=blue,
urlcolor=blue]{hyperref}

\newtheorem{theorem}{Theorem}[section]
\newtheorem{lemma}{Lemma}[section]

\newtheorem{conjecture}{Conjecture}[section]

\newtheorem{proposition}{Proposition}[section]
\theoremstyle{definition}

\newtheorem{definition}{Definition}[section]
\theoremstyle{remark}

\newcommand{\cH}{\mathcal{H}}
\newcommand{\cS}{\mathcal{S}}
\newcommand{\cA}{\mathcal{A}}
\newcommand{\cB}{\mathcal{B}}

\newcommand{\cF}{\mathcal{F}}
\newcommand{\cC}{\mathcal{C}}

\newcommand{\cE}{\mathcal{E}}
\newcommand{\cP}{\mathcal{P}}
\newcommand{\sF}{\mathscr{F}}

\DeclareMathOperator{\exco}{exco_2}
\DeclareMathOperator{\co}{co_2}

\tikzset{
vtx/.style={inner sep=1.1pt, outer sep=0pt, circle, fill,draw}, 
vtxl/.style={inner sep=1.1pt, outer sep=0pt, rectangle, fill=red,draw=black}, 
hyperedge/.style={fill=gray,opacity=0.3,draw=black}, 
}

\title{Some exact and asymptotic results for hypergraph Tur\'an problems in $\ell_2$-norm}
\author{George Brooks\thanks{University of South Carolina, Columbia, SC. ({\tt ghbrooks@email.sc.edu})}
\and
William Linz\thanks{University of South Carolina, Columbia, SC. ({\tt wlinz@mailbox.sc.edu}). Partially supported by NSF RTG Grant DMS 2038080.}}
\date{\today}

\begin{document}
\maketitle

\begin{abstract}
For a $k$-uniform hypergraph $\cH$, the \emph{codegree squared sum} $\co(\cH)$ is the square of the $\ell_2$-norm of the codegree vector of $\cH$, and for a family $\sF$ of $k$-uniform hypergraphs, the codegree squared extremal number $\exco(n, \sF)$ is the maximum codegree squared sum of a hypergraph on $n$ vertices which does not contain any hypergraph in $\sF$. Balogh, Clemen and Lidick\'y recently introduced the codegree squared extremal number and determined it for a number of $3$-uniform hypergraphs, including the complete graphs $K_4^3$ and $K_5^3$. 

In this paper, we give a number of exact or asymptotic results for hypergraph Tur\'an problems in the $\ell_2$-norm, including the first exact results for arbitrary $k$. Namely, we prove a version of the classical Erd\H{o}s-Ko-Rado theorem for the codegree squared extremal number: if $\cF \subset \binom{[n]}{k}$ is intersecting and $n\ge 2k$, then \[\co(\cF) \le \binom{n-1}{k-1}(1+(n-k+1)(k-1)),\]
with equality only for the star for $n > 2k$. Our main tool is an inequality of Bey, which also gives a general upper bound on $\exco(n, \sF)$. 

We also prove versions of the Erd\H{o}s Matching Conjecture and the $t$-intersecting Erd\H{o}s-Ko-Rado theorem for the codegree squared extremal number for large $n$, determine the exact codegree squared extremal number of minimal and linear $3$-paths and $3$-cycles, and determine asymptotically the codegree squared extremal number of minimal and linear $s$-paths and $s$-cycles for $s\ge 4$. 

Lastly, we derive a number of exact or asymptotic results for graph Tur\'an-type problems in the $\ell_2$-norm from spectral extremal results for certain forbidden subgraph problems and the well-known Hofmeister's inequality.   

\end{abstract}

\section{Introduction}
\subsection{Hypergraph Tur\'an problems in the \texorpdfstring{$\ell_2$}{}-norm}

A $k$-uniform hypergraph is a pair $\cH = (V, \cE)$, with a set of vertices $V$ and a set of edges $\cE$ such that for each $E\in \cE$, $E\subset V$ and $|E| = k$. We use the notation $[n] = \{1, 2, 3, \ldots, n\}$ for a positive integer $n$ and $\binom{V}{k}$ for the family of $k$-element subsets of $V$, so that $\cE \subset \binom{V}{k}$. If $|V| = n$, we typically replace $V$ by $[n]$. Given a family of $k$-uniform hypergraphs $\sF$, a $k$-uniform hypergraph $\cH$ is \emph{$\sF$-free} if it does not contain a copy of any member of $\sF$. 

Given a family $\sF$ of $k$-uniform hypergraphs, the \emph{extremal number} $\text{ex}(n, \sF)$ is the maximum number of edges in a $k$-uniform $\sF$-free hypergraph on $n$ vertices. The \emph{Tur\'an density} $\pi(\sF)$ is the scaled limit 
\[\pi(\sF) = \lim_{n\rightarrow \infty} \frac{\text{ex}(n, \sF)}{\binom{n}{k}}.\]

Problems about extremal numbers and Tur\'an densities are among the most central and well-studied of extremal combinatorics; see the survey of Keevash~\cite{Keev}. 

Recently, Balogh, Clemen and Lidick\'y~\cite{BCL1, BCL2} introduced a new type of extremal number for hypergraphs based on the $\ell_2$-norm of the codegree vector of a hypergraph.

\begin{definition}[Codegree vector and $\ell_2$-norm]\label{defn:cdvecl2}
Let $\cH \subset \binom{[n]}{k}$ be a $k$-uniform hypergraph on $[n]$. For a set $E\subset [n]$, the \textit{codegree} of $E$, $d(E)$, is the number of edges in $\cH$ containing $E$. The \textit{codegree vector} is the vector ${\bf x} \in \mathbb{Z}^{\binom{[n]}{k-1}}$ with entries given by ${\bf x}_{\{v_1, \ldots, v_{k-1}\}} = d(\{v_1, \ldots, v_{k-1}\})$ for all $(k-1)$-sets $\{v_1, \ldots, v_{k-1}\} \subset [n]$. The \emph{codegree squared sum} $\co(\cH)$ is the square of the $\ell_2$-norm of the codegree vector of $\cH$, \textit{i.e.}
\[\co(\cH) = \sum_{E\in \binom{[n]}{k-1}} d(E)^2.\]
\end{definition}

Note that for a $k$-uniform hypergraph $\cH$, $\sum_{E\in \binom{[n]}{k-1}}d(E) = k|\cH|$, so the $\ell_1$-norm of the codegree vector corresponds to the classical extremal number up to the constant $k$.

For a family $\sF$ of $k$-uniform hypergraphs, Balogh, Clemen and Lidick\'y define $\exco(n, \sF)$ to be the maximum codegree squared sum among all $k$-uniform $\sF$-free hypergraphs with $n$ vertices and further define the \emph{codegree squared density} $\sigma(\sF)$ to be the scaled limit of $\exco(\sF)$, so that
\[\sigma(\sF) =\lim_{n\rightarrow \infty}\frac{\exco(n, \sF)}{\binom{n}{k-1}{(n-k+1)^2}}.\]

\subsection{Kleitman-West problem}
The Kleitman-West problem is the following discrete edge isoperimetric problem: given positive integers $n > k > 0$ and a positive integer $0\le m\le \binom{n}{k}$, which hypergraph $\cF \subset \binom{[n]}{k}$ with $|\cF| = m$ maximizes the number of pairs $\{F, F'\}$ with $F, F'\in \cF$ and $|F\cap F'| = k-1$? Recall that the Johnson graph $J(n, k)$ is the graph with vertex set $\binom{[n]}{k}$ and edges $AB \Leftrightarrow |A\cap B| = k-1$. For a set $\cF \subset V(J(n, k))\cong \binom{[n]}{k}$, let $e_{J(n, k)}(\cF, \cF) = \{\{F, F'\}:\, F, F'\in \cF,\, FF'\in E(J(n, k))\}$ be the number of edges in the subgraph of the Johnson graph induced by the family $\cF$. The Kleitman-West problem is equivalent to: given $0\le m\le \binom{[n]}{k}$, what is the maximum size of $e_{J(n, k)}(\cF, \cF)$ over all $k$-uniform hypergraphs with $|\cF| = m$? 

The case $k=2$ was solved by Ahlswede and Katona~\cite{AK1978}, who proved that the maximum size families are either quasi-complete graphs or quasi-star graphs. Kleitman made a similar conjecture for larger $k$, that the first $m$ sets in either the lexicographic order or colexicographic order maximizes $e_{J(n, k)}(\cF, \cF)$. Ahlswede and Cai~\cite{AC1999} gave a counterexample to this conjecture for $k=3$. Gruslys, Letzter and Morrison~\cite{GLM} gave a smaller counterexample to the Kleitman-West question when $n=7$, $k=3$ and $m=11$. Namely, Gruslys, Letzter and Morrison take
\[\cH = \{123, 124, 125, 126, 127, 134, 135, 136, 145, 146, 156\} \subset \binom{[7]}{3}.\]
By computer, we found similar counterexamples when $n=9$, $k=4$ and $m=10$. Take
\[\cS = \{1234, 1236, 1238, 1239, 1346, 1348, 1349, 1368, 1369, 1389\} \subset \binom{[9]}{4}.\]

On the other hand, Das, Gan and Sudakov~\cite[Theorem 1.8]{DGS2016} solved the Kleitman-West problem in the case where $n$ is very large and the number of edges $m$ is very small or very large. Harper~\cite{Har91} solved the Kleitman-West problem for certain values of $m$ by using a continuous relaxation; see also \cite[Section 10.2]{Har04}. 

We show that hypergraph Tur\'an problems in the $\ell_2$-norm and the Kleitman-West problem are closely related. Let $\cH$ be a non-$k$-partite $k$-uniform hypergraph. Let $e(\cH, \cH)$ denote the maximum number of edges in an induced subgraph of $J(n, k)$ over all sets of vertices $V \subseteq V(J(n, k)) \cong \binom{[n]}{k}$ which do not contain a copy of $\cH$. The Tur\'an-type problem of maximizing the $\ell_2$-norm over all $k$-uniform hypergraphs on $n$ vertices which are $\cH$-free is asymptotically equivalent to finding $e(\cH, \cH)$. 

\begin{theorem}\label{thm:excokwthm}
Let $\cH$ be a non-$k$-partite $k$-uniform hypergraph. Then, 
\[\exco(n, \cH)= 2e(\cH, \cH) + O(n^k).\]
\end{theorem}

This theorem follows from a simple counting lemma (Lemma~\ref{lem:reform}) for the codegree squared sum, which is proved in Section 2. 

\subsection{Results}

A family of sets $\cF$ is \emph{$t$-intersecting} if for any two sets $F, F'\in \cF$, $|F\cap F'|\ge t$. When $t=1$, we say the family is \emph{intersecting}. The classic Erd\H{o}s-Ko-Rado theorem determines the maximum size of a $k$-uniform intersecting family. 

\begin{theorem}[Erd\H{o}s-Ko-Rado~\cite{EKR}]\label{ekrthm}
Let $\cF \subset \binom{[n]}{k}$ be an intersecting family with $n \ge 2k$. Then, 
\[|\cF| \le \binom{n-1}{k-1}.\]
If $n > 2k$, equality holds only if $\cF \cong \{F\in \binom{[n]}{k}: 1\in F\}$. 
\end{theorem}

One of the main theorems of this paper is a version of the Erd\H{o}s-Ko-Rado theorem in the $\ell_2$-norm.

\begin{theorem}[Erd\H{o}s-Ko-Rado in $\ell_2$-norm]\label{fullekrl2}
Let $\cF \subset \binom{[n]}{k}$ be an intersecting family with $n\ge 2k$. Then,

\[\co(\cF) \le \binom{n-1}{k-1}(1+(n-k+1)(k-1)),\]
with equality only if $\cF \cong \{F\subset \binom{[n]}{k} : 1\in F\}$ if $n > 2k$.  If $n=2k$, equality holds only for the $1$-star $\cF \cong \{F\subset \binom{[2k]}{k} : 1\in F\}$ and the complement of the $1$-star $\cF \cong\ \binom{[2k-1]}{k}$. 
\end{theorem}

Note in the classic Erd\H{o}s-Ko-Rado theorem, when $n=2k$ every maximal intersecting family is extremal, so there are $2^{\binom{2k-1}{k-1}}$ extremal families, while there are only $4k$ extremal families for the Erd\H{o}s-Ko-Rado theorem in the $\ell_2$-norm. 

Theorem~\ref{fullekrl2} turns out to be an immediate consequence of an inequality on the codegree squared sum of a hypergraph proven by Bey~\cite{Bey}, which extends an inequality for graphs due to de Caen~\cite{dC1998}. In fact, as we shall observe, Bey's inequality immediately shows that any Tur\'an-type problem in which the $1$-star is the extremal construction in the $\ell_1$-norm also has the $1$-star as the extremal construction in the $\ell_2$-norm. We also give several other applications of Bey's inequality to Tur\'an-type problems in the $\ell_2$-norm. We give Bey's inequality in Section 2 and note that it is a consequence of the expander mixing lemma.

The \textit{matching number} $\nu(\cF)$ of a $k$-uniform hypergraph $\cF$ is the maximum number of pairwise disjoint sets in $\cF$. The $k$-uniform hypergraphs $\cF$ with $\nu(\cF) = 1$ are the intersecting families. Erd\H{o}s~\cite{E1965} proved that for $n$ sufficiently large, the maximum size of a $k$-uniform hypergraph $\cF$ on vertices with matching number $\nu(\cF) = s$ is the following family $\cB(n, k, s)$:

\[\cB(n, k, s):= \Bigg\{ F\in \binom{[n]}{k}: F\cap [s]\neq \emptyset \Bigg\}.\]

Erd\H{o}s also made a conjecture for the maximum-size $k$-uniform family with matching number $s$ for all $n \ge ks$. This is now known as the Erd\H{o}s Matching Conjecture, and remains unproven for the full range of $n$.  Even the full range in which $\cB(n, k, s)$ is the maximum-size $k$-uniform hypergraph with matching number $s$ has not been determined. The best known bounds on this range are:  $n\ge (2s+1)k-s$ for all $s$ by Frankl~\cite{F2013}; and $n\ge \frac53sk - \frac23s$ for sufficiently large $s$ by Frankl and Kupavskii~\cite{FK2022}.

For sufficiently large $n$, we show that $\cB(n, k, s)$ also has maximum $\ell_2$-norm among all $k$-uniform hypergraphs with matching number $s$. The proof is given in Section 4.  

\begin{theorem}[Erd\H{o}s Matching Conjecture in $\ell_2$-norm]\label{thm:emcl2}
Let $\cF\subset \binom{[n]}{k}$ be a family of $k$-sets with $\nu(\cF)\le s$. There exists an integer $n_0(k, s)$ such that for $n \ge n_0(k, s)$, 
\[\co(\cF) \le \co(\cB(n, k, s)),\]
with equality if and only if $\cF \cong \cB(n, k, s)$.
\end{theorem} 

There are several notions of paths and cycles in hypergraphs.  A \emph{Berge $s$-cycle} is a $k$-uniform hypergraph with $s$ edges $E_1, \ldots, E_s$  such that there are $s$ distinct vertices $v_1, \ldots, v_s$ with $v_{i} \in E_i \cap E_{i+1}$ for $1\le i\le s-1$ and $v_{s} \in E_s\cap E_1$. A \emph{minimal $s$-cycle} is a Berge $s$-cycle with edges $E_1, \ldots, E_s$ such that $E_i \cap E_j \neq\emptyset$ if and only if $|j-i|=1 \text{ or } \{i, j\} = \{1, s\}$ and no vertex belongs to each edge $E_1, \ldots, E_s$. The \emph{linear $s$-cycle} is the $k$-uniform hypergraph with $s$ edges $E_1, \ldots, E_s$ such that \[|E_i \cap E_j| = \begin{cases}1 & \text{if } |j-i|=1 \text{ or } \{i, j\} = \{1, s\} \\ 0 &\text{otherwise.}\end{cases}\] 

We denote the family of all minimal $k$-uniform $s$-cycles by $\cC_s^k$ and the linear $k$-uniform $s$-cycle by $C_s^k$. A $k$-uniform \emph{Berge $s$-path} (\emph{minimal $s$-path}, \emph{linear $s$-path}, respectively) is a hypergraph obtained from a $k$-uniform Berge (minimal, linear, respectively) $s$-cycle by deleting one of the edges. We denote the family of all minimal $k$-uniform $s$-paths by $\cP_s^k$ and the linear $k$-uniform $s$-path by $P_s^k$. 

We can determine the codegree squared extremal number exactly for minimal $3$-paths and $3$-cycles $\cP_3^k$ and $\cC_3^k$ and also for the linear $3$-path and the linear $3$-cycle. 

\begin{theorem}\label{thm:3cyc}
Let $k \ge 4$ be an integer. Then, for sufficiently large $n$, the codegree squared extremal number of the linear path $P_3^k$ satisfies
\[\exco(n, P_3^k) = \binom{n-1}{k-1}(1+(n-k+1)(k-1)).\]
 For $k\ge 3$, and for sufficiently large $n$, the codegree squared extremal number of  the linear cycle $C_3^k$ satisfies
 \[ \exco(n, C_3^k) = \binom{n-1}{k-1}(1+(n-k+1)(k-1)).\]
Similarly, for any $k\ge 3$ and for $n \ge 2k$ the  codegree squared extremal number of the family of minimal $k$-paths $\cP_3^k$ is given by
\[\exco(n, \cP_3^k) = \binom{n-1}{k-1}(1+(n-k+1)(k-1)),\]  
and for any $k\ge 3$ and $n \ge \frac{3k}{2}$, the codegree squared extremal number of the family of minimal $k$-cycles $\cC_3^k$ is given by 
\[\exco(n, \cC_3^k) = \binom{n-1}{k-1}(1+(n-k+1)(k-1)).\]
In all cases, equality holds only for the hypergraph $\cH \cong \{F\subset \binom{[n]}{k}: 1\in F\}$.
\end{theorem}

The case $\exco(n, C_3^3)$ (where $n\ge 6$ suffices) was previously proved by Balogh, Clemen and Lidick\'y~\cite{BCL1}. 

For $s\ge 4$, we can obtain the leading order term of the codegree squared extremal number. 

\begin{theorem}\label{thm:kcyc}
Let $s$ and $k$ be integers, with $s\ge 4$ and $k\ge 3$. Then, the codegree squared extremal numbers of the linear path $P_s^k$ and the linear cycle $C_s^k$ satisfy
\[\exco(n, P_s^k) = \Big\lfloor{\frac{s-1}{2}\Big\rfloor}k(k-1)\binom{n}{k}(1+o(1)),\]
\[ \exco(n, C_s^k) = \Big\lfloor{\frac{s-1}{2}\Big\rfloor}k(k-1)\binom{n}{k}(1+o(1)).\]
Similarly, the extremal codegree squared numbers of the family of minimal $k$-paths $\cP_s^k$ and the family of minimal $k$-cycles $\cC_s^k$ satisfy 
\[\exco(n, \cP_s^k) =  \Big\lfloor{\frac{s-1}{2}\Big\rfloor}k(k-1)\binom{n}{k}(1+o(1)), \]
and 
\[\exco(n, \cC_s^k) = \Big\lfloor{\frac{s-1}{2}\Big\rfloor}k(k-1)\binom{n}{k}(1+o(1)). \]
\end{theorem}

The families $\cB(n, k, \lfloor{\frac{s-1}{2}\rfloor})$ show that the main term is tight. The cases $\exco(n, C_s^3)$ and $\exco(n, P_s^3)$ were previously proven by Balogh, Clemen and Lidick\'y~\cite{BCL1}.

We obtain other results in Sections 4 and 5, including a version of the $t$-intersecting Erd\H{o}s-Ko-Rado theorem in the $\ell_2$-norm and a general upper bound on $\co(\cH)$ for any non-$k$-partite $k$-uniform hypergraph. 

In Section 6, we turn our attention to forbidden subgraph problems in the $\ell_2$-norm and show how several asymptotic and extremal results for such problems can be obtained from the corresponding spectral extremal result. For a graph $G$, the \emph{spectral radius} $\lambda_1$ is the largest eigenvalue of the adjacency matrix of $G$. The question of determining the graph on $n$ vertices with maximum spectral radius satisfying some properties was pioneered by Nikiforov~\cite{Nik2002} with his spectral Tur\'an theorem.

We can deduce many extremal results for forbidden subgraph problems in the $\ell_2$-norm from the corresponding spectral extremal result. As a representative example, we can easily prove the following general theorem. 

\begin{theorem}\label{thm:kkn-kgenthm}
Suppose that the extremal graph for a spectral Tur\'an problem forbidding the family of graphs $\sF$ is $K_{k, n-k}$. Then, $K_{k, n-k}$ is also an extremal graph for the corresponding forbidden subgraph problem in the $\ell_2$-norm. 

Suppose that the extremal graph for a spectral Tur\'an problem forbidding the family of graphs $\sF$ is $K_{k} \vee \overline{K_{n-k}}$ or $K_k \vee (\overline{K_{n-k-2}} \cup K_2)$. Then, 
\[\exco(H) = kn^2(1+o(1)).\]
\end{theorem}

As recently shown by Byrne, Desai and Tait~\cite{BDT}, there are many Tur\'an-type problems where the spectral extremal graph is one of $K_{k, n-k}$, $K_k \vee \overline{K_{n-k}}$ or $K_k \vee (\overline{K_{n-k-2}} \cup K_2)$. As one example, Cioab\u{a}, Desai and Tait~\cite{CDT} recently proved a spectral version of the Erd\H{o}s-S\'os conjecture that was originally conjectured by Nikiforov~\cite{Nik2010}. Namely, Cioab\u{a}, Desai and Tait proved that for sufficiently large $n$, if $k\ge 2$ and $G$ is a graph of order $n$ such that $\lambda_1(G) \ge \lambda_1(K_k\vee \overline{K_{n-k}})$, then $G$ contains all trees of order $2k+2$ unless $G\cong K_k\vee \overline{K_{n-k}}$; and similarly if $k\ge 2$ and $G$ is a graph of order $n$ such that $\lambda_1(G) \ge \lambda_1(K_k \vee (\overline{K_{n-k-2}} \cup K_2))$, then $G$ contains all trees of order $2k+3$ unless $G\cong K_k \vee (\overline{K_{n-k-2}} \cup K_2)$. Theorem~\ref{thm:kkn-kgenthm} implies an asymptotic version of the Erd\H{o}s-S\'os conjecture for the $\ell_2$-norm: if $\sF$ is the set of all graphs which contain all trees on $k$ vertices for some $k\ge 6$, then $\exco(n, \sF) = \left(\lfloor{\frac{k}{2}\rfloor}-1\right) n^2(1+o(1))$.

\section{Useful lemmas}
\subsection{A counting lemma for the codegree squared sum}
We give a useful reformulation of the codegree squared sum for a $k$-uniform hypergraph. 

\begin{lemma}\label{lem:reform}
For any $\cF \subseteq \binom{[n]}{k}$,
\[
\co(\cF) = k|\cF|+ 2 \left|\left\{\{F, F'\} \in \binom{\cF}{2} : |F\cap F'| = k -1\right\}\right|.
\]
\end{lemma}
\begin{proof}
We observe
\begin{align*}
    \text{co}_2(\cF) &= \sum_{E\in \binom{[n]}{k-1}}d(E)^2\\
    &= \sum_{E\in \binom{[n]}{k-1}} \left(\sum_{F\in \cF} \mathbbm{1}_{E \subseteq F}\right) \left(\sum_{F'\in \cF} \mathbbm{1}_{E \subseteq F'}\right)\\
    &= \sum_{E\in \binom{[n]}{k-1}} \left(\sum_{(F, F')\in \cF \times \cF} \mathbbm{1}_{E \subseteq F \cap F'}\right)\\
    &= \sum_{(F, F')\in \cF \times \cF}  \left(\sum_{E\in \binom{[n]}{k-1}}\mathbbm{1}_{E \subseteq F \cap F'}\right)\\
    &= \sum_{(F, F')\in \cF \times \cF} \left( k \cdot \mathbbm{1}_{F = F'} + \mathbbm{1}_{|F\cap F'| = k-1} \right)\\
    &= k|\cF| +  2 \left|\left\{\{F, F'\} \in \binom{\cF}{2} : |F\cap F'| = k -1 \right\}\right|.
\end{align*}
\end{proof}

\subsection{Connections to Kleitman-West}
 Lemma~\ref{lem:reform} implies Theorem~\ref{thm:excokwthm}, namely that the Tur\'an-type problem of maximizing the $\ell_2$-norm over all non-$k$-partite $k$-uniform hypergraphs on $n$ vertices which are $\cH$-free is asymptotically equivalent to finding $2e(\cH, \cH)$. 

\begin{proof}[Proof of Theorem~\ref{thm:excokwthm}]
    Balogh, Clemen and Lidick\'y~\cite[Proposition 1.11]{BCL2} showed that the inequality $\sigma(\cH)>0$ holds if and only if $\cH$ is not $k$-partite (generalizing a classic result of Erd\H{o}s~\cite{E1964}). Let $\cF$ be a hypergraph which maximizes $\text{exco}_2(\cH)$. Lemma~\ref{lem:reform} shows that \[\text{co}_2(\cF) = k|\cF| + 2e(\cF, \cF) = 2e(\cF, \cF) + O(n^k),\]
    as $\text{co}_2(\cF) = \Theta(n^{k+1})$, but $k|\cF| = \Theta(n^{k})$. 
\end{proof}

The asymptotic equivalence between $\exco(\cH)$ and $2e(\cF, \cF)$ implied by Theorem~\ref{thm:excokwthm} for non-$k$-partite $k$-uniform hypergraphs is not true in general for $k$-partite $k$-uniform hypergraphs. Let $\cH$ be the $k$-uniform hypergraph consisting of two $k$-edges which intersect in $k-1$ vertices. Any $\cH$-free $k$-uniform hypergraph $\cF$ corresponds to an independent set in the Johnson graph $J(n, k)$, so $e(\cF, \cF) = 0$, whence $\exco(n, \cH) = k\,\text{ex}(n, \cH)$. Note $\text{ex}(n, \cH) = \Theta(n^{k-1})$ since $\chi(J(n, k))\le n$\cite{GS}. 

Lemma~\ref{lem:reform} also implies that for dense $k$-uniform hypergraphs the Kleitman-West problem is asymptotically equivalent to finding the maximum codegree squared sum over all hypergraphs with a fixed edge-density. More precisely, let $\alpha \in (0, 1]$ and consider all $k$-uniform hypergraphs such that $|\cF| = \alpha\binom{n}{k}$. Then for any such hypergraph $\cF$, 
\[
\text{co}_2(\cF) = k|\cF| + 2e(\cF, \cF) = 2e(\cF, \cF) + O(n^k).
\]

\subsection{de Caen's inequality, Bey's inequality and expander mixing}

Let $G$ be a graph on $n$ vertices with degree sequence $d_1 \ge d_2 \ge \cdots \ge d_n$. De Caen~\cite{dC1998} proved the following upper bound on the sum $\sum_{1\le i\le n} d_i^2$.

\begin{theorem}[de Caen]\label{thm:dcthm}
Let $G = (V, E)$ be a graph on $n$ vertices and $e$ edges. Then, 
\[\sum_{i=1}^nd_i^2 \le e\left(\frac{2e}{n-1} + n-2\right).\]
\end{theorem}

Bey~\cite{Bey} generalized Theorem~\ref{thm:dcthm} to $k$-uniform hypergraphs. 
\begin{theorem}[Bey~\cite{Bey}]\label{thm:bey}
Let $\cH = (V, \cE)$ be a $k$-uniform hypergraph with $|V| = n$. Let $\ell$ be an integer satisfying $0\le \ell \le k$. Then,
\[\sum_{E\in \binom{V}{\ell}} d(E)^2 \le \frac{\binom{k}{\ell}\binom{k-1}{\ell}}{\binom{n-1}{\ell}}|\cE|^2 + \binom{k-1}{\ell-1}\binom{n-\ell-1}{k-\ell}|\cE|.\]
In particular, if $\ell = k-1$, then 
\[\sum_{E\in \binom{V}{k-1}}d(E)^2\le \frac{k}{\binom{n-1}{k-1}}|\cE|^2 + (k-1)(n-k)|\cE|.\]
Equality holds if and only $\cH$ is one of the following hypergraphs: the $1$-star $S_k^1 = \{F\in \binom{[n]}{k}: 1 \in F\}$ or the complement of the $1$-star; the complete graph $K_k^n$ or its complement; or if $n=k+1$ and $\cH$ is the $t$-star $S_k^t = \{F \in \binom{[k+1]}{k}: [t] \subset F\}$ for $2\le t\le \lfloor{\frac{k+1}{2}\rfloor}$ or its complement.
\end{theorem}

We show that both Theorem~\ref{thm:dcthm} and the $\ell=k-1$ case of Theorem~\ref{thm:bey} can be proved by Lemma~\ref{lem:reform} and (a slight generalization of) the expander mixing lemma (see \cite[Lemma 4.15]{Vad}). Bey gives a similar proof in \cite{Bey2}, where he also notes the relation between Theorem~\ref{thm:bey} and the Kleitman-West problem. 

 A graph $G$ on $n$ vertices is an \emph{$(n, d, \lambda)$-graph} if $G$ is $d$-regular and all eigenvalues apart from $d$ are at most $\lambda$ in absolute value. For two sets $S, T \subset V(G)$, let $E(S, T)$ be the number of edges with one endpoint in $S$ and one in $T$ (edges with endpoints in $S\cap T$ are counted twice).

\begin{theorem}[Expander mixing lemma]\label{thm:eml}
Let $G$ be an $(n, d, \lambda)$-graph.  Then, 
\[\left|E(S, T) - \frac{d|S||T|}{n}\right| \le \lambda\sqrt{|S||T|\left(1-\frac{|S|}{n}\right)\left(1-\frac{|T|}{n}\right)}. \]
\end{theorem}


Lemma~\ref{lem:reform} implies that 
\[\sum_{E\in \binom{[n]}{k-1}}d(E)^2 = k|\cF| + 2e(\cF, \cF) = k|\cF| + E(\cF, \cF).\]

\begin{proof}[Proofs of Theorems~\ref{thm:dcthm} and \ref{thm:bey}]
We bound $E(\cF, \cF)$ by the expander mixing lemma. The Johnson graph $J(n, k)$ has $d = k(n-k)$ and $\lambda = (k-1)(n-k) - k$, so 
\begin{align*}
    E(\cF, \cF) &\le \frac{k(n-k)|\cF|^2}{\binom{n}{k}} + ((k-1)(n-k)-k)|\cF|\left(1-\frac{|\cF|}{\binom{n}{k}}\right)\\
    & = \frac{n}{\binom{n}{k}}|\cF|^2 + ((k-1)(n-k) - k)|\cF|\\
    & = \frac{k}{\binom{n-1}{k-1}}|\cF|^2 + ((k-1)(n-k) - k)|\cF|
\end{align*}
Lemma~\ref{lem:reform} now implies
\[\text{co}_2(\cF) \le k|\cF| + \frac{k}{\binom{n-1}{k-1}}|\cF|^2 + ((k-1)(n-k) - k)|\cF| = \frac{k}{\binom{n-1}{k-1}}|\cF|^2 + (k-1)(n-k)|\cF|\]
\end{proof}

We record the codegree squared sums of the $t$-star $S_k^t$ and the family $\cB(n, k, s)$.

\begin{lemma}[Codegree squared sum of the $t$-star]\label{lem:starcd}
Let $S_k^t$ be the $t$-star, that is, $S_k^t = \{F \in \binom{[n]}{k}: [t] \subset F\}$. Then, 
\[\co(S_k^t) = \binom{n-t}{k-t}(t + (n-k+1)(k-t)).\]
\end{lemma}

\begin{lemma}[Codegree squared sum of $\cB(n, k, s)$]\label{lem:bcd}
We have
\[\co(\cB(n, k, s)) = s^2\binom{n-s}{k-1} + (n-k+1)^2\left(\binom{n}{k-1} - \binom{n-s}{k-1}\right).\]
\end{lemma}

\section{Exact results in the \texorpdfstring{$\ell_2$}{}-norm from Bey's inequality}

A number of hypergraph Tur\'an problems have either the star or $\cB(n, k, s)$ as the exact or asymptotic extremal example. Theorem~\ref{thm:bey} allows us to convert these results in the $\ell_1$-norm to either exact or asymptotic results in the $\ell_2$-norm.  

\subsection{\texorpdfstring{Erd\H{o}s}{}-Ko-Rado}

Theorem~\ref{fullekrl2} is an immediate consequence of Theorems~\ref{ekrthm} and~\ref{thm:bey}.

\begin{proof}[Proof of Theorem~\ref{fullekrl2}]
Let $\cF \subset \binom{[n]}{k}$ with $n\ge 2k$. If $\cF$ is intersecting, then by the Erd\H{o}s-Ko-Rado theorem, $|\cF| \le \binom{n-1}{k-1}$, so Theorem~\ref{thm:bey} implies 
\begin{align*}
\co(\cF) &\le \frac{k}{\binom{n-1}{k-1}}\binom{n-1}{k-1}^2 + (k-1)(n-k)\binom{n-1}{k-1} \\
&= \binom{n-1}{k-1}(k + (k-1)(n-k)) \\
& = \binom{n-1}{k-1}(1 + (k-1)(n-k+1)),
\end{align*}
which by Lemma~\ref{lem:starcd} is the codegree-squared sum of the $1$-star. Equality holds for $n > 2k$ only for the $1$-star, as this is the only case of equality in the Erd\H{o}s-Ko-Rado theorem. If $n=2k$, equality only holds for the $1$-star or its complement by the equality statement of Theorem~\ref{thm:bey}.
\end{proof}

In fact, as Bey~\cite{Bey2} observed, Theorem~\ref{thm:bey} implies that the star maximizes $e_{J(n, k)}(\cF, \cF)$ when $|\cF| = \binom{n-1}{k-1}$.  

\subsection{Paths and cycles}

Theorems~\ref{thm:3cyc} and \ref{thm:kcyc} follow from Theorem~\ref{thm:bey} and the known results on the extremal numbers of $P_s^k$, $C_s^k$, $\cP_s^k$ and $\cC_s^k$. We first state the exact results for $\text{ex}(n, P_3^k)$ and $\text{ex}(n, C_3^k)$. 

\begin{theorem}[Linear $3$-paths and $3$-cycles]\label{thm:lin3paths3cycles}
Let $k\ge 4$ be an integer. Then, for $n$ sufficiently large, 
\[\text{ex}(n, P_3^k) = \binom{n-1}{k-1},\]
with equality only for the $1$-star. Similarly, for $k\ge 3$ and $n$ sufficiently large,
\[\text{ex}(n, C_3^k) = \binom{n-1}{k-1},\]
with equality only for the $1$-star. 
\end{theorem} 
Frankl and F\"{u}redi~\cite{FF1987} determined the extremal examples for $\text{ex}(n, C_3^k)$ for large $n$ and F\"{u}redi, Jiang and Seiver~\cite{FJS2014} determined the extremal examples for $\text{ex}(n, P_3^k)$, $n$ large and $k\ge 4$. Cs\'ak\'any and Kahn~\cite{CK1999} showed $\text{ex}(n, C_3^3) = \binom{n-1}{2}$ for $n\ge 6$. 

The exact results for $\text{ex}(n, \cP_3^k)$~\cite{MV1} and $\text{ex}(n, \cC_3^k)$~\cite{MV2} were proven by Mubayi and Verstra\"{e}te.

\begin{theorem}[Minimal $3$-paths and $3$-cycles]\label{thm:min3paths3cycles}
Let $k\ge 3$ be an integer. Then, for $n \ge 2k$,
\[\text{ex}(n, \cP_3^k) = \binom{n-1}{k-1},\]
with equality only for the $1$-star. Similarly, for $k\ge 3$ and $n \ge \frac32 k$,
\[\text{ex}(n, \cC_3^k) = \binom{n-1}{k-1},\]
with equality only for the $1$-star. 
\end{theorem} 

\begin{proof}[Proof of Theorem~\ref{thm:3cyc}]
Theorem~\ref{thm:bey} as used in the proof of the Erd\H{o}s-Ko-Rado theorem in the $\ell_2$-norm and the extremal number results in Theorems~\ref{thm:lin3paths3cycles} and~\ref{thm:min3paths3cycles} imply the claims. 
\end{proof}

For $s\ge 4$, the exact results for large $n$ are known for $\text{ex}(n, P_s^k)$ and for $\text{ex}(n, \cP_s^k)$ by the work of F\"{u}redi, Jiang and Seiver~\cite{FJS2014} and Kostochka, Mubayi and Verstra\"{e}tre~\cite{KMV}. The exact results for large $n$ were determined for $\text{ex}(n, C_s^k)$ and for $\text{ex}(n, \cC_s^k)$ by F\"{u}redi and Jiang~\cite{FJ2014} and by Kostochka, Mubayi and Versta\"{e}te~\cite{KMV}. We only need the main term of the result. 

\begin{theorem}[Minimal and linear $s$-paths and $s$-cycles]\label{thm:mlinspathsscycles}
Let $s$ and $k$ be integers with $s\ge 4$ and $k\ge 3$ and set $\ell = \lfloor{\frac{s-1}{2}\rfloor}$. Then, for $n$ sufficiently large, 
\[\text{ex}(n, P_s^k) = \text{ex}(n, C_s^k) = \text{ex}(n, \cP_s^k) = \text{ex}(n, \cC_s^k) = \binom{n}{k} - \binom{n-\ell}{k} + O(n^{k-2}).\]
\end{theorem}

\begin{proof}[Proof of Theorem~\ref{thm:kcyc}]
By Theorem~\ref{thm:mlinspathsscycles}, for large $n$, $|\cF| \le \binom{n}{k} - \binom{n-\ell}{k}  = \ell k\frac{n^{k-1}}{k!}(1 + o(1))$, where $\cF$ is an $\cH$-free hypergraph for some $\cH \in \{P_s^k, C_s^k, \cP_s^k, \cC_s^k\}$. Therefore, Theorem~\ref{thm:bey} implies 
\begin{align*}
\co(\cF) & \le \frac{k}{\binom{n-1}{k-1}}\left(\ell k\frac{n^{k-1}}{k!}(1 + o(1))\right)^2 + (k-1)(n-k)\ell k\frac{n^{k-1}}{k!}(1 + o(1))\\
& = \ell k(k-1)\binom{n}{k}(1 + o(1)).
\end{align*}
Lemma~\ref{lem:bcd} implies that $\co(\cB(n, k, \ell)) = \ell k(k-1)\binom{n}{k}(1+o(1))$. 
\end{proof}

\subsection{Set systems with no cluster or no simplex-cluster}
Let $d\ge 1$ be a positive integer. Let $\cA:= \{A_1, A_2, \ldots, A_{d+1}\} \subset \binom{[n]}{k}$ be a family of $d+1$ $k$-sets over the ground set $[n]$. The family $\cA$  is a \emph{$d$-simplex} if $\cap_{i=1}^{d+1}A_i = \emptyset$, but for each $1\le j\le d+1$, $\cap_{i\neq j}A_i \neq \emptyset$. The family $\cA$ is a \emph{$d$-cluster} if $\cap_{i=1}^{d+1}A_i = \emptyset$ and $\left|\cup_{i=1}^{d+1}A_i\right|\le 2k$. If $\cA$ is both a $d$-simplex and a $d$-cluster, then $\cA$ is a $d$-simplex-cluster. 

Note that a $1$-simplex and a $1$-cluster are both equivalent to an intersecting family, and a $2$-simplex is equivalent to a minimal $3$-cycle. Chv\'atal~\cite{Ch} conjectured that the $1$-star is the maximum size $k$-uniform family with no $d$-simplex for all $n\ge (d+1)k/d$. This was proved for $n\ge n_0(k, d)$ by Frankl and F\"{u}redi~\cite{FF1987}. Mubayi~\cite{Mub2006} conjectured that the $1$-star is also the maximum size $k$-uniform family with no $d$-cluster for $n\ge (d+1)k/d$. Keevash and Mubayi~\cite{KM} further conjectured that the $1$-star is also the maximum size $k$-uniform family with no $d$-simplex-cluster. This last conjecture was proved for $n\ge n_0(d)$ by Lifshitz~\cite{Lif2020}.

The maximum size of a $k$-uniform family $\cF$ with no $d$-cluster was completely solved by Currier~\cite{Cur1}. 

\begin{theorem}[\cite{Cur1}]\label{thm:dclustthm}
Let $n, k, d$ be positive integers with $1\le d\le k$ and $n\ge \frac{d+1}{d}k$. Suppose $\cF \subset \binom{[n]}{k}$ does not contain a $d$-cluster. Then, 
\[|\cF|\le \binom{n-1}{k-1},\]
and except for the case $d=1$ and $n=2k$ equality holds only if $\cF \cong S_k^1$. 
\end{theorem}

Currier also obtained the current best result for families containing no $d$-simplex-cluster. 

\begin{theorem}[\cite{Cur2}]\label{thm:dsimpclustthm}
Let $n, k, d$ be positive integers with $4\le d+1\le k$ and $n\ge 2k-d+2$. Suppose $\cF \subset \binom{[n]}{k}$ does not contain a $d$-simplex-cluster. Then, 
\[|\cF| \le \binom{n-1}{k-1},\]
and equality holds only if $\cF \cong S_k^1$. 
\end{theorem}

Note that any family which contains no $d$-simplex-cluster also contains no $d$-simplex. 

Theorem~\ref{thm:bey} combined with Theorems~\ref{thm:dclustthm} and~\ref{thm:dsimpclustthm} shows that stars are also extremal for families without simplices, clusters, or simplex-clusters. 

\begin{theorem}[Set systems with no $d$-cluster in $\ell_2$-norm]\label{thm:dclustthmell2}
Let $n, k, d$ be positive integers with $1\le d\le k$ and $n\ge \frac{d+1}{d}k$. Suppose $\cF \subset \binom{[n]}{k}$ does not contain a $d$-cluster. Then, 
\[\co(\cF) \le \binom{n-1}{k-1}(1+(n-k+1)(k-1)),\]
and except for the case $d=1$ and $n=2k$ equality holds only if $\cF \cong \cS_k^1$. If $d=1$ and $n=2k$, equality holds only if $\cF \cong S_k^1$ or $\cF\cong \binom{[2k-1]}{k}$. 
\end{theorem}

Note that in the case $d=1$ and $n=2k$ there are only two cases of equality, as in Theorem~\ref{fullekrl2}. 

\begin{theorem}[Set systems with no $d$-simplex-cluster in $\ell_2$-norm]\label{thm:dsimpclustthmell2}
Let $n, k, d$ be positive integers with $4\le d+1\le k$ and $n\ge 2k-d+2$. Suppose $\cF \subset \binom{[n]}{k}$ does not contain a $d$-simplex-cluster. Then, 
\[\co(\cF) \le \binom{n-1}{k-1}(1+(n-k+1)(k-1)),\]
and equality holds only if $\cF \cong S_k^1$. 
\end{theorem}

\section{\texorpdfstring{Erd\H{os}}{} Matching Conjecture and \texorpdfstring{$t$-intersecting}{} \texorpdfstring{Erd\H{o}s}{}-Ko-Rado}

In this section, we use Theorem~\ref{thm:bey} to prove versions of the Erd\H{o}s Matching Conjecture  and the $t$-intersecting Erd\H{o}s-Ko-Rado theorem in the $\ell_2$-norm for sufficiently large $n$. We only prove the theorems for very large $n$, but expect the statements should hold for close to the range of $n$ for the $\ell_1$-norm in both of these settings.

\subsection{\texorpdfstring{Erd\H{o}s}{} Matching Conjecture in \texorpdfstring{$\ell_2$}{}-norm}

We first observe that Theorem~\ref{thm:bey} implies the main term in the Erd\H{o}s Matching Conjecture for $n$ sufficiently large by the same argument as the proof for Theorem~\ref{thm:kcyc}. 

\begin{proposition}\label{prop:maintermemc}
Let $n, k, s$ be positive integers, and let $\cF \subset \binom{[n]}{k}$ be a $k$-uniform hypergraph with $\nu(\cF) \le s$. Then, for $n$ sufficiently large,
\[\co(\cF) \le sk(k-1)\binom{n}{k}(1+o(1)),\]
and the family $\cB(n, k, s)$ shows that the main term is tight. 
\end{proposition}

Unfortunately, the lower order term in the bound from Theorem~\ref{thm:bey} is larger than the lower order term in $\co(\cB(n, k, s))$. To obtain an exact result, we use the following stability result for the Erd\H{o}s Matching Conjecture in the $\ell_1$-norm proven by Frankl and Kupavaskii~\cite{FK}. Recall the \emph{covering number} $\tau(\cF)$ of a family $\cF$ is the smallest size of a set $A \subset [n]$ such that $A\cap F\neq\emptyset$ for every $F\in \cF$. 

\begin{theorem}[Frankl-Kupavskii]\label{thm:fkthm}
Fix integers $s, k \ge 2$. Let $n = (u+s-1)(k-1) + s + k$, $u\ge s+1$. Then, for any family $\cF \subset \binom{[n]}{k}$ with $\nu(\cF) = s$ and $\tau(\cF) \ge s+1$, 
\[|\cF| \le \binom{n}{k} - \binom{n-s}{k} - \frac{u-s-1}{u}\binom{n-s-k}{k-1}.\]
\end{theorem}

Note that $k$-uniform hypergraphs with $\nu(\cF) = s$ and $\tau(\cF) \ge s+1$ are those which are not a subfamily of $\cB(n, k, s)$. We now prove Theorem~\ref{thm:emcl2}. 

\begin{proof}[Proof of Theorem~\ref{thm:emcl2}]
If $\cF \subset \cB(n, k, s)$, then the claim is immediate, so assume $\cF \not\subset \cB(n, k, s)$. Then, by Theorem~\ref{thm:fkthm}, for $u$ (and $n$) sufficiently large, $|\cF| \le \binom{n}{k} - \binom{n-s}{k} - \frac12\binom{n-s-k}{k-1}$, so as in the proof of Proposition~\ref{prop:maintermemc}, Theorem~\ref{thm:bey} implies for $n$ sufficiently large
\[\co(\cF) \le \left(\left(s-\frac12\right)k(k-1)\right)\binom{n}{k}(1+o(1)).\]
\end{proof} 

\subsection{\texorpdfstring{$t$-intersecting}{} \texorpdfstring{Erd\H{o}s}{}-Ko-Rado in \texorpdfstring{$\ell_2$}{}-norm}

Erd\H{o}s, Ko and Rado~\cite{EKR} proved that the $t$-stars $S_k^t$ are maximum-size $t$-intersecting $k$-uniform families for sufficiently large $n$. Wilson~\cite{Wil} proved the optimal range on $n$ for all $t$ and $k$. 

\begin{theorem}[$t$-intersecting Erd\H{o}s-Ko-Rado]\label{thm:tintekr}
Let $n, k$ and $t$ be positive integers with $k > t$. Let $\cF\subset \binom{[n]}{k}$ be a $t$-intersecting family. Then, for $n\ge (t+1)(k-t+1)$, 
\[|\cF| \le \binom{n-t}{k-t}.\]
Equality holds for $n > (t+1)(k-t+1)$ only if $\cF \cong \{F \in \binom{[n]}{k}: [t] \subset F\}.$
\end{theorem}

Theorem~\ref{thm:bey} is insufficient to determine the maximum size $k$-uniform $t$-intersecting family in the $\ell_2$-norm. We additionally need the $t$-intersecting version of the Hilton-Milner theorem, a strong stability result for the $t$-intersecting Erd\H{o}s-Ko-Rado theorem.

A family of sets $\cF$ is \emph{nontrivial $t$-intersecting} if $\cF$ is $t$-intersecting and $|\cap_{F\in \cF}F| < t$. The nontrivial $t$-intersecting families are those which are not subfamilies of the $t$-stars. We introduce two particular nontrivial $t$-intersecting families $\cA(n, k, t)$ and $\cH(n, k, t)$. 

\[\cH(n, k, t):= \{F\subset \binom{[n]}{k}: [t] \subset F, F\cap [t+1, k+1]\neq \emptyset\} \cup \{[k+1]\setminus \{i\}: i \in [t]\}.\]

\[\cA(n, k, t):= \{F \subset \binom{[n]}{k}: |F\cap [t+2]| \ge t+1\}.\]

In the case $t=1$, Hilton and Milner determined the optimal nontrivial intersecting families. 

\begin{theorem}[Hilton-Milner~\cite{HM}]\label{hmthm}
Let $\cF \subset \binom{[n]}{k}$ be a nontrivial intersecting family with $n > 2k$. Then, 
\[|\cF| \le \binom{n-1}{k-1} - \binom{n-k-1}{k-1} + 1.\]
If $k > 3$, equality holds only if $\cF \cong \cH(n, k, 1)$. If $k=3$, equality only holds if $\cF\cong \cH(n, 3, 1)$ or $\cF\cong \cA(n, 3, 1)$.
\end{theorem}

The $t$-intersecting Hilton-Milner theorem was proved by Frankl~\cite{F1978} for large $n$ and by Ahlswede and Khachatrian~\cite{AK2} for the optimal range $n > (t+1)(k-t+1)$. 

\begin{theorem}[$t$-intersecting Hilton-Milner]\label{thm:inthm}
Let $n, k, t$ be integers and let $\cF \subset \binom{[n]}{k}$ be a nontrivial $t$-intersecting family. If $n > (t+1)(k-t+1)$, then 
\[|\cF| \le \max\{|\cA(n, k, t)|, |\cH(n, k, t)|\}.\]
Furthermore, equality holds only if $\cF \cong \max\{\cA(n, k, t), \cH(n, k, t)\}$, $k > 2t+1$, or $\cF \cong \cA(n, k, t)$, $k \le 2t+1$. 
\end{theorem}

It is easy to verify that for large $n$
\[|\cA(n, k, t)| \sim (t+2)\binom{n}{k-t-1};\, |\cH(n, k, t)| \sim (k-t+1)\binom{n}{k-t-1}.\]


\begin{proposition}\label{prop:easyubekr}
Let $n, k, t$ be integers with $ k > t > 0$ and let $\cF \subset \binom{[n]}{k}$ be a $t$-intersecting family. Then, there is an integer $n_0(k, t)$ such that for $n \ge n_0(k, t)$, 

\[\co(\cF) \le \binom{n-t}{k-t}(1+(n-k+1)(k-t)),\]
with equality only if $\cF \cong \{F\in \binom{[n]}{k} : [t]\subset F\}$. 
\end{proposition}

\begin{proof}
The statement is immediate if $\cF \subset S_k^t$. Suppose that $\cF$ is a nontrivial $t$-intersecting family. Theorem~\ref{thm:bey} and Theorem~\ref{thm:inthm} imply that for large enough $n$, 
\[\co(\cF) \le \max(t+2, k-t+1)\cdot(k-1)(k-t)\binom{n}{k-t} + O(n^{k-t-1}),\]
which is much smaller than $\binom{n-t}{k-t}(t+(n-k+1)(k-t))$ for large $n$. 
\end{proof}

A family of sets $\cF$ is \text{$d$-wise $t$-intersecting} if for any $d$ sets $F_1, F_2\ldots, F_d \in \cF$, $|F_1\cap \cdots \cap F_d| \ge t$. A proof similar to the one for Proposition~\ref{prop:easyubekr} using the analogous Hilton-Milner theorem for ($d$-wise) $t$-intersecting families~\cite{BL, OV} shows that the $t$-stars also have maximum $\ell_2$-norm among all $d$-wise $t$-intersecting families for sufficiently large $n$. 

\begin{proposition}\label{prop:easyubdtekr}
Let $\cF \subset \binom{[n]}{k}$ be a $d$-wise $t$-intersecting family for integers $d\ge 2$ and $t\ge 1$. Then, there is an integer $n_0(k, t, d)$ such that if $n \ge n_0(k, t, d)$, we have
\[\co(\cF) \le \binom{n-t}{k-t}(1+(n-k+1)(k-t)),\]
with equality only if $\cF \cong \{F\in \binom{[n]}{k} : [t]\subset F\}$. 
\end{proposition}

In the case $t=1$, Frankl~\cite{F1976} proved that for $n \ge dk/(d-1)$, the maximum-size $d$-wise intersecting family is the star $S_k^1$. Theorem~\ref{thm:bey} extends Frankl's result to the $\ell_2$-norm as an immediate corollary. 

\begin{theorem}\label{thm:dwiseintl2}
Let $\cF \subset \binom{[n]}{k}$ be a $d$-wise intersecting family with $n\ge \frac{d}{d-1}\,k$. Then, 
\[\co(\cF) \le \binom{n-1}{k-1}(1+(n-k+1)(k-1)). \]
\end{theorem}

We conjecture that the $t$-stars have maximum codegree squared sum in the same range that they are maximal in the classical $t$-intersecting Erd\H{o}s-Ko-Rado theorem. 

\begin{conjecture}\label{conj:exacttint}
Let $\cF \subset \binom{[n]}{k}$ be a $t$-intersecting family for an integer $t\ge 1$. Then, if $n \ge (t+1)(k-t+1)$, we have
\[\co(\cF) \le \binom{n-t}{k-t}(1+(n-k+1)(k-t)),\]
with equality for $n > (t+1)(k-t+1)$ only if $\cF \cong \{F\in \binom{[n]}{k} : [t]\subset F\}$. 
\end{conjecture}

It would be interesting to prove a version of the Hilton-Milner theorem for the $\ell_2$-norm. We first handle the case $k=3$. 

\begin{proposition}\label{prop:3unifprop}
    Let $\cF \subset \binom{[n]}{3}$ be a nontrivial intersecting family with $n\ge 7$. Then, 
    \[\text{co}_2(\cF) \le \text{co}_2(\cH(n, 3, 1)), \]
    and equality holds if and only if $\cF \cong \cH(n, 3, 1)$ or $\cF \cong \cA(n, 3, 1)$. 
\end{proposition}

\begin{proof}
We first compute $\text{co}_2(\cH(n, 3, 1))$ and $\text{co}_2(\cA(n, 3, 1))$. For $\cH(n, 3, 1)$, the $2$-sets $\{1, 2\}$, $\{1, 3\}$, and $\{1, 4\}$ have degree $n-2$, the sets $\{1, a\}$, $5\le a\le n$ have degree $3$, the sets $\{2, 3\}, \{2, 4\}, \{3, 4\}$ have degree $2$, the sets $\{2, a\}, \{3, a\}, \{4, a\}$, $5\le a\le n$ each have degree $1$  and the sets $\{a, b\} \subset \binom{[5, n]}{2}$ each have degree $0$. Thus, 
\[\text{co}_2(\cH(n, 3, 1)) = 3(n-2)^2 + 3^2(n-4) + 3(1)^2(n-4) + 2^2(3) = 3n^2-24.\]

For $\cA(n, 3, 1)$, the sets $\{1, 2\}, \{1, 3\}, \{2, 3\}$ each have degree $n-2$, the sets $\{1, a\}, \{2, a\}, \{3, a\}$, $4\le a\le n$ each have degree $2$, and the sets $\{a, b\} \subset \binom{[4, n]}{2}$ each have degree $0$, so
\[\text{co}_2(\cA(n, 3, 1)) = 3(n-2)^2+3(2)^2(n-3)=3n^2-24.\]

Polycn and Ruci\'nski~\cite[Theorem 4]{PR} classified all maximal intersecting $3$-graphs for $n\ge 7$. We can exhaustively check all of the possible families to see that all other maximal $3$-uniform intersecting families have smaller codegree squared sum than $\cH(n, 3, 1)$ and $\cA(n, 3, 1)$ for $n\ge 7$. 
\end{proof}

 For $k\ge 4$, we conjecture that $\cH(n, k, 1)$ has the maximum $\ell_2$-norm among all nontrivial intersecting families. 

\begin{conjecture}\label{con:hm}
Let $\cF \subset \binom{[n]}{k}$ be a nontrivial intersecting family. Then, for $k\ge 3$ and $n > 2k$, 
\[\co(\cF) \le \co(\cH(n, k, 1)),\]
with equality if and only if $\cF \cong \cH(n, k, 1)$ if $k\ge 4$, and if and only if $\cF \cong \cH(n, 3, 1)$ or $\cF\cong \cA(n, 3, 1)$ if $k=3$.

If $n=2k$, then 
\[\co(\cF) \le k^2\binom{2k-1}{k-1},\]
and equality holds only if $\cF \cong \binom{[2k-1]}{k}$. 
\end{conjecture} 

The case $n=2k$ is implied by the $\ell_2$-norm Erd\H{o}s-Ko-Rado theorem. We can prove Conjecture~\ref{con:hm} for large $n$ by using Bey's inequality and known stability results for the Hilton-Milner theorem. We omit the details, as the goal should be to prove the conjecture for $n > 2k$.

\section{An upper bound on \texorpdfstring{$\sigma$}{} for general hypergraphs}

Balogh, Clemen and Lidick\'y~\cite{BCL2} proved that $\sigma(\cF) \le \pi(\cF)$ for any $k$-uniform hypergraph $\cF$. Theorem~\ref{thm:bey} implies a general upper bound on $\sigma(\cF)$ for any $k$-uniform hypergraph $\cF$ in terms of the Tur\'an density $\pi(\cF)$, which is an improvement over the previous bound when $0 < \pi(\cF) < 1$. 

\begin{theorem}[General bound on $\sigma(\cF)$]\label{thm:genl2upperbd}
Let $\cF$ be a $k$-uniform hypergraph with Tur\'an density $\pi(\cF) > 0$. Then, we have that 
\[\sigma(\cF) \le \pi(\cF) \left(\frac{\pi(\cF)}{k} + 1 - \frac{1}{k}\right).\]
\end{theorem}

\begin{proof}[Proof of Theorem~\ref{thm:genl2upperbd}]
By definition, $|\cF| \le (\pi(\cF)+o(1))\binom{n}{k}$. Therefore, Theorem~\ref{thm:bey} implies
\[
\sum_{E\in \binom{V(\cF)}{k-1}}d(E)^2 \le \frac{k}{\binom{n-1}{k-1}}\left((\pi(\cF)+o(1))\binom{n}{k}\right)^2 + (k-1)(n-k)(\pi(\cF)+o(1))\binom{n}{k}.
\]
Dividing through by $\binom{n}{k-1}(n-k+1)^2$ and using that $\binom{n}{k} \sim \frac{n^k}{k!}$ for $k$ fixed and $n\rightarrow \infty$, we obtain
\[\sigma(\cF) \le \frac{1}{k}\pi(\cF)^2 + \frac{k-1}{k}\pi(\cF) =  \pi(\cF)\left(\frac{\pi(\cF)}{k} + 1 - \frac{1}{k}\right).\]
\end{proof}

 For specific hypergraphs $\cF$, this upper bound seems to be rather weak, because it is a general upper bound on the Kleitman-West problem for all hypergraphs $\cF$ with $|\cF| = (\pi(\cF)+o(1))\binom{n}{k}$.  Balogh, Clemen and Lidick\'y~\cite{BCL1, BCL2} used flag algebras to determine $\sigma$ asymptotically for a number of $3$-uniform hypergraphs, and upper bounds for other hypergraphs. The bounds given by Theorem~\ref{thm:genl2upperbd} are much worse than the bounds obtained by flag algebras computations.  On the other hand, Theorem~\ref{thm:genl2upperbd} works for any non-$k$-partite $k$-uniform hypergraph. 

Let $\mathbb{F}$ be the Fano plane. Balogh, Clemen and Lidick\'y were unable to use flag algebras to improve over the trivial upper bound for $\sigma(\mathbb{F})$.  Using the fact that $\pi(\mathbb{F}) = \frac34$~\cite{dCF}, Theorem~\ref{thm:genl2upperbd} implies the following upper bound for $\sigma(\mathbb{F})$. 

\begin{proposition}\label{prop:fano}
\[\sigma(\mathbb{F}) \le \frac{11}{16}.\]
\end{proposition}

Let $K_{t}^k$ be the complete $k$-uniform hypergraph on $t$ vertices. The best known general upper bound on $\pi(K_{t}^k)$ was proved by de Caen~\cite{dC1983}. 

\begin{theorem}[de Caen~\cite{dC1983}]\label{thm:genkt}
For any integers $t > k \ge 2$, 
\[\pi(K_{t}^k) \le 1 - \frac{1}{\binom{t-1}{k-1}}.\]
\end{theorem}

Theorems~\ref{thm:genl2upperbd} and~\ref{thm:genkt} immediately give a general upper bound on $\sigma(K_t^k)$. 

\begin{proposition}
For integers $t > k \ge 2$, 
\[\sigma(K_t^k) \le \left(1-\frac{1}{\binom{t-1}{k-1}}\right)\left(1-\frac{1}{k\binom{t-1}{k-1}}\right).\]
\end{proposition}

\section{Some Tur\'an-type results for graphs in the \texorpdfstring{$\ell_2$}{}-norm}

The inequalities of Bey and de Caen give bounds on the codegree squared sum $\co(\cF)$ solely in terms of the number of edges in $\cF$, which apart from cases like the $1$-star limits their use in Tur\'an problems for the codegree squared sum absent additional information, such as stability results for the corresponding Tur\'an problem in the $\ell_1$-norm. In this section, we show that in the graph case ($k=2$), the corresponding spectral extremal problem can allow us to quickly deduce exact or asymptotic results for the corresponding Tur\'an-type problem in the $\ell_2$-norm.  

Any upper bound on the spectral radius $\lambda_1$ for a graph $G$ also provides an upper bound for the number of edges $m$ in the graph $G$ via the well-known inequality $\lambda_1 \ge 2m/n$. However, the spectral radius also provides an upper bound on $\co(G)$, as proved by Hofmeister~\cite{Hof}. 

\begin{theorem}[Hofmeister]\label{thm:hofmineq}
Let $G$ be a graph on $n$ vertices with spectral radius $\lambda_1$. Then, 
\[\co(G) \le n\lambda_1^2.\]
\end{theorem}

Using Hofmeister's inequality, we can prove Theorem~\ref{thm:kkn-kgenthm}. 

\begin{proof}[Proof of Theorem~\ref{thm:kkn-kgenthm}]
Let $\sF$ be a family of graphs, such that $K_{k, n-k}$ maximizes $\lambda_1$ over all $\sF$-free graphs. Then, note that 
\[\co(K_{k, n-k}) = k(n-k)^2 + (n-k)k^2 = nk(n-k),\]
while $\lambda_1 = \sqrt{k(n-k)}$. Hofmeister's inequality now implies immediately that \[\exco(\sF) = nk(n-k)\] and $K_{k,n-k}$ is an extremal graph in the $\ell_2$-norm. 

Now, let $H$ be one of $K_k \vee \overline{K_{n-k}}$ or $K_k \vee (\overline{K_{n-k-2}} \cup K_2)$, and let $\sF$ be a family of graphs such that $H$ maximizes $\lambda_1$ over all $\sF$-free graphs. Observe that 
\[\co(H) = kn^2 + O(n),\]
while it can be shown (see \cite{Nik2010}) that
\[\lambda_1(K_k \vee \overline{K_{n-k}}) = (k-1)/2 + \sqrt{kn - (3k^2+2k-1)/4}\]
and 
\[\lambda_1(K_k \vee (\overline{K_{n-k-2}} \cup K_2)) = \lambda_1(K_k \vee \overline{K_{n-k}}) + O\left(\frac{1}{n+\sqrt{n}}\right).\]

Thus, in either case 
\[\lambda_1^2(H) = kn + O(\sqrt{n}),\]
so Hofmeister's inequality and the putative spectral Tur\'an result imply
\[\exco(\sF) = kn^2(1+o(1)).\]
\end{proof}

In fact, Theorem~\ref{thm:kkn-kgenthm} can be extended to determine asymptotic Tur\'an-type results for the sums of the $p$th powers of degrees of graphs. For a family of graphs $\sF$, let $t_p(n, \sF):= \max\{\sum_{i\in V(G)} d_i^p: \text{$G$ is a $\sF$-free graph on $n$ vertices.}\}$, where for a graph $G$, $d_v$ is the degree of vertex $v$ in $G$. Caro and Yuster~\cite{CY2000, CY2004} introduced the problem of determining $t_p(n, \sF)$ for different families of graphs $\sF$.  

\begin{theorem}\label{thm:pthpowers}
Suppose that the extremal graph for a spectral Tur\'an problem forbidding the family of graphs $\sF$ is one of $K_{k, n-k}$, $K_{k} \vee \overline{K_{n-k}}$ or $K_k \vee (\overline{K_{n-k-2}} \cup K_2)$. Then, for any $p \ge 2$, 
\[t_p(n, \sF) = kn^p(1+o(1)).\]
\end{theorem}

\begin{proof}[Proof of Theorem~\ref{thm:pthpowers}]
Let $H$ be one of $K_{k, n-k}$, $K_{k} \vee \overline{K_{n-k}}$ or $K_k \vee (\overline{K_{n-k-2}} \cup K_2)$. We have
\[\sum_{i\in V(H)}d_i^p = kn^p(1+o(1)).\]
By Theorem~\ref{thm:kkn-kgenthm}, we have $\exco(\sF) = kn^2(1+o(1))$, so for any graph $G$ which is $\sF$-free, 
\[\sum_{i\in V(G)}d_i^p = \sum_{i\in V(G)}d_i^{p-2}d_i^2 \le n^{p-2}\sum_{i\in V(G)}d_i^2 \le kn^p(1+o(1)),\]
completing the proof. 
\end{proof}

Caro and Yuster~\cite{CY2000} conjectured that $t_p(n, C_{2k}) = kn^p(1+o(1))$. Nikiforov~\cite{Nik2009} proved this conjecture, while Gerbner~\cite{Gerb1} recently gave a different short proof of the conjecture. Cioab\u{a}, Desai and Tait~\cite{CDT1} recently proved that $K_k \vee (\overline{K_{n-k-2}} \cup K_2)$ has maximum spectral radius over all graphs on $n$ vertices without a $C_{2k}$ for $k\ge 3$, so Theorem~\ref{thm:pthpowers} gives another proof of the conjecture of Caro and Yuster (the precise extremal graph in the case $k=2$ is not known for all $n$, but the known bounds~\cite{Nik2007} give the same asymptotic result for $t_p(n, C_4)$). 

As mentioned in the Introduction, Theorem~\ref{thm:kkn-kgenthm} is useful because there are many Tur\'an-type problems where the spectral extremal graph is one of $K_{k, n-k}$, $K_k \vee \overline{K_{n-k}}$ or $K_k \vee (\overline{K_{n-k-2}} \cup K_2)$. Byrne, Desai and Tait~\cite{BDT} proved a general spectral extremal theorem which captures many of the known forbidden subgraph problems where one of those graphs are the spectral extremal result for the Tur\'an-type problem. We list a few such Tur\'an-type results in the $\ell_2$-norm; more can be deduced from the results in the paper of Byrne, Desai and Tait. 

\begin{proposition}[Paths]\label{prop:paths}
For any $k\ge 1$, we have
\[\exco(n, P_{2k+2}) = kn^2(1+o(1))\]
and
\[\exco(n, P_{2k+3}) = kn^2(1+o(1)).\]
\end{proposition}

Nikiforov~\cite{Nik2010} proved for $k \ge 1$, the $P_{2k+2}$-free graph with maximum spectral radius for large $n$ is $K_{k} \vee \overline{K_{n-k}}$; similarly, the $P_{2k+3}$-free graph with maximum spectral radius is $K_k\vee (\overline{K_{n-k-2}} \cup K_2)$. 

\begin{proposition}[Disjoint cycles]\label{prop:disjointcycles}
Let $k\ge 2$ and let $\sF$ be the set of all disjoint unions of $k$ cycles. Then, 
\[\exco(n, \sF) = (2k-1)n^2(1+o(1)).\]
\end{proposition}

 Erd\H{o}s and P\'osa~\cite{ErP} proved that the $\sF$-free graph with maximum number of edges is $K_{2k-1}\vee \overline{K_{n-2k+1}}$. Recently, Liu and Zhai~\cite{LZ22} proved that $K_{2k-1}\vee \overline{K_{n-2k+1}}$ is also the $\sF$-graph with maximum spectral radius. 

\begin{proposition}[$K_r$-minor-free graphs]\label{prop:krminorfree}
For a given $k\ge 3$, let $\sF$ be the set of graphs with a $K_k$-minor. Then,
\[\exco(n, \sF)=(k-2)n^2(1+o(1)).\]
\end{proposition}

Tait~\cite{Tait2019} proved that $K_{k-2}\vee \overline{K_{n-k+2}}$ is the $K_k$-minor-free graph with maximum spectral radius for large $n$. 

Finally, we mention that other asymptotic Tur\'an-type results in the $\ell_2$-norm can be deduced from known spectral Tur\'an theorems and Hofmeister's inequality. 

\begin{proposition}[Outerplanar and planar graphs]\label{prop:colindeverdiere}
Let $\sF$ be the family of all graphs which contain a $K_4$-minor or a $K_{2, 3}$-minor (so that the family of $\sF$-free graphs is the family of outerplanar graphs). Then, 
\[\exco(n, \sF) = n^2(1+o(1)).\]
Similarly, let $\mathscr{G}$ be the family of all graphs which contain a $K_5$-minor or a $K_{3, 3}$-minor (so that the family of $\mathscr{G}$-free graphs is the family of planar graphs). Then, 
\[\exco(n, \mathscr{G}) = 2n^2(1+o(1)).\]
\end{proposition}

  Tait and Tobin~\cite{TT2017} showed that the graph $K_1 \vee P_{n-1}$ is the outerplanar graph with maximum spectral radius for large $n$; furthermore, the graph $K_2 \vee P_{n-2}$ is the planar graph with maximum spectral radius. These results were generalized by Tait~\cite{Tait2019}, who showed that $K_{r-1}\vee P_{n-r+1}$ is the graph on $n$ vertices for large $n$ with maximum spectral radius and Colin de Verdi\'ere number at most $r$ (the graphs with Colin de Verdi\'ere number at most $2$ are the outerplanar graphs, and the graphs with Colin de Verdi\'ere number at most $3$ are the planar graphs).

\begin{proposition}[$K_{s, t}$-minor-free graphs]\label{prop:kstminorfree}
Let $t\ge s\ge 2$, and let $\sF$ be the family of all graphs which contain a $K_{s, t}$-minor (so that the family of $\sF$-free graphs is the family of $K_{s, t}$-minor-free graphs). Then,
\[\exco(n, \sF) = (s-1)n^2(1+o(1)).\]
\end{proposition}

Tait~\cite{Tait2019} proved that for large $n$, any $K_{s, t}$-minor-free graph on $n$ vertices satisfies
\[\lambda_1 \le \frac{s+t-3+\sqrt{(s+t-3)^2+4((s-1)(n-s+1)-(s-2)(t-1))}}{2}.\]

The precise $K_{s, t}$-minor-free graphs with maximum spectral radius were subsequently determined for all $s$ and $t$ by Zhai and Lin~\cite{ZL2022}. The graph $K_{s-1}\vee (\frac{n-s+1}{t}K_t)$ gives the asymptotically tight lower bound in Proposition~\ref{prop:kstminorfree}.

\end{document}